\documentclass[10pt]{article}
  \usepackage{amsmath,amssymb}
  \usepackage{cite}
  \usepackage{latexsym}% defines $\Box$ for LaTeX2e
  \usepackage{tikz}
  \usepackage{setspace}

  \newcommand{\cP}{\mathcal{P}}
  \newcommand{\cQ}{\mathcal{Q}}

  \newcommand{\hs}{\hspace*{\parindent}}

  \newcommand{\qed}{\hspace*{\fill} $\Box$\\}

  \newtheorem{theo}{\bfseries \hs Theorem}[section]
  
  \newtheorem{prop}[theo]{\bfseries \hs Proposition}
  \newtheorem{lemma}[theo]{\bfseries \hs Lemma}
  \newtheorem{corol}[theo]{\bfseries \hs Corollary}

  \numberwithin{equation}{section} % Automatically number equations within sections

\begin{document}

\title{Nonmedian Direct Products of Graphs with Loops}

\author{Kristi Clark\thanks{College of Information and Mathematical Sciences, Clayton State University, (kclark14@student.clayton.edu)} \hbox{ and} Elliot Krop\thanks{College of Information and Mathematical Sciences, Clayton State University, (ElliotKrop@clayton.edu).}}
   \date{\today}

\maketitle

\begin {abstract}
A \emph{median graph} is a connected graph in which, for every three vertices, there exists a unique vertex $m$ lying on the geodesic between any two of the given vertices. We show that the only median graphs of the direct product $G\times H$ are formed when $G=P_k$, for any integer $k\geq 3$ and $H=P_l$, for any integer $l\geq 2$, with a loop at an end vertex, where the direct product is taken over all connected graphs $G$ on at least three vertices or at least two vertices with at least one loop, and connected graphs $H$ with at least one loop. 
\\[\baselineskip] 2000 Mathematics Subject
      Classification: 05C75, 05C76, 05C12
\\[\baselineskip]
\end {abstract}

 \section{Introduction}
 
 For basic graph theoretic notation and definition see Diestel \cite{Diest}.  Classification of median graphs and the study of graph products have been active fields of study for the last several decades, see \cite{Bandelt}, \cite{KlavzarMulder}, \cite{BJKZ}, as well as the book \cite{ImrichKlavzar}. 
 
 In \cite{BJKZ}, the authors categorized whether components of direct product graphs are median. Along with some simpler cases, they showed that if $G$ is $P_3$ and $H$ is an even tree (a tree with even distance between every two vertices of degree three or more), then one component of $G \times H$ is a median graph. The classification of which product graphs $P_2 \times H$ are median is only partly understood.
 
 The graphs in question above have no loops. As for the case with loops, let $T_n$, the \emph{total graph}, be the complete graph on $n$ vertices with a loop at each vertex. Munarini showed in \cite{Munarini}, that $G \times T_n$ is median if and only if $G=P_2$ and $n=2$.
 
 In this paper, we consider the rest of the cases, namely, direct products of graphs where some factor graph has at least one loop.\\
 
 We say two vertices $x$ and $y$ are adjacent if there exists an edge $(x,y)$ and in this case we write $x \sim y$. If $x$ and $y$ are not adjacent we write $x \not \sim y$. By $N(x,y)$ we mean the \emph{deleted neighborhood of x and y}, that is, the set of vertices $\{v: v\sim x, v\sim y, v\neq x, v\neq y\}$. We define the distance between vertices $x$ and $y$ in $G$ as the length of a shortest path from $x$ to $y$ in $G$, and denote it $d_G(x,y)$. The \emph{direct product} of two graphs $G$ and $H$, written as $G\times H$ is the graph with vertices $V(G\times H)=V(G)\times V(H)$ where two vertices $(v_1,h_1)$ and $(v_2,h_2)$ are adjacent if and only if $v_1$ and $v_2$ are adjacent in $G$ as well as $h_1$ and $h_2$ are adjacent in $H$. A \emph{median graph} is a connected graph in which for every three vertices, there exists a unique vertex $m$ known as the \emph{median}, which lies on the geodesic between any two of the given vertices. A \emph{cartesian product} of two graphs $G$ and $H$, written as $G\Box H$ is the graph with vertices $V(G\times H)=V(G)\times V(H)$ where two vertices $(v_1,h_1)$, $(v_2,h_2)$ are adjacent if $v_1=v_2$ in $G$ and $h_1 \sim h_2$ in $H$, or $v_1 \sim v_2$ in $G$ and $g_1=g_2$ in $G$. For any positive integers $n$ and $m$, call $P_n \Box P_m$ the \emph{complete grid graph}. A \emph{grid graph} is a subgraph of a complete grid graph.
 
 A subgraph $H$ of $G$ is called \emph{isometric} if $d_H(x,y)=d_G(x,y)$ for all $x,y \in V(H)$. A subgraph $H$ of $G$ is called \emph{convex} if it is connected and for every pair of vertices $x,y\in V(H)$, every isometric path from $x$ to $y$ exists in $H$.\\
 
 In this paper, we consider the direct product $G\times H$, where $G$ is either a connected graph on at least three vertices or a connected graph with at least one loop, and $H$ is a connected graph on at least two vertices with at least one loop. We show that the only median graphs of such a direct product are formed when $G=P_k$ and $H = P_l$ with a loop at an end vertex, for any integers $k\geq 3, l\geq 2$.

 \section{Preliminaries}
 
 The first fact is folklore.
 
 \begin{prop}\label{findoddcycle}
 For the paths $\cP_0, |\cP_0|=k$, and $\cQ_0, |\cQ_0|=l$, beginning at vertex $u$ and ending at vertex $v$, if $k$ and $l$ are of opposite parity, then $\cP_0\cup \cQ_0$ contains an odd cycle.
 \end{prop}

 %\begin{proof}
 %Order the vertices of $\cP_0$ beginning at $u$ and ending at $v$, and let $v_1$ be the first vertex in the list, not $v$, so that $v_1\in \cP_0\cap \cQ_0$. Let $d_0=d_{\cP_0}(u,v_1)$ and $e_0=d_{\cQ_0}(u,v_1)$. If the parity of $d_0$ is the opposite of $e_0$, then the trail from $u$ along $\cP_0$ to $v_1$ and back to $u$ along $\cQ_0$ is an odd cycle. If the parity of $d_0$ is the same as that of $e_0$, then let $\cP_1$ be the subpath of $\cP_0$ beginning at $v_1$ and ending at $v$ and $\cQ_1$ be the subpath of $\cQ_0$ beginning at $v_1$ and ending at $v$. Repeat the argument with $\cP_i$ in place of $\cP_{i-1}$ and $\cQ_i$ in place of $\cQ_{i-1}$, for $i\geq 1$.
 
 %Let $v_m$ be the last vertex, not $v$ so that $v_m \in \cP_0 \cap \cQ_0$. Since $\sum_{i=0}^{m}d_i=k$ and $\sum_{i=0}^{m}e_i=l$, we can find an integer $j$ between $0$ and $m$, so that the parity of $d_j$ is different from the parity of $e_j$.
 
 %\qed
 %\end{proof}

 We state a well known lemma which may be found in \cite{ImrichKlavzar}.
 
 %\begin{lemma} \label{distance}
 %Let $x_1, x_2, x_3$ be vertices of a graph $G$. If they have a median $z$, then $d(x,z)=\frac{1}{2}(d(x_i,x_j)+d(x_i,x_k)-d(x_j,x_k))$ where $\{i,j,k\}=\{1,2,3\}$
 
 %\end{lemma}
 
 %This leads to the next result.
 
 \begin{lemma} \label{bipartite}
 Median graphs are bipartite.

 \end{lemma}
 
 %\vspace{.2 in}
 
 The following subgraph is excluded from median graphs by definition.
 
 \begin{prop}\label{K23}
 
 If $G$ contains $K_{2,3}$ as an induced subgraph, then $G$ is not median.
 
 \end{prop}
 
 More generally, by using the the above lemma that median graphs are bipartite, we have
 
 \begin{prop}\label{K23sub}
 If $K_{2,3}$ is a subgraph of $G$ (not necessarily induced), then $G$ cannot be a median graph.
 \end{prop}
 
 The next observation allows us to exclude the case where both $G$ and $H$ have loops.
 
 \begin{prop} \label{loops in G and H}
 If $G$ and $H$ are connected graphs on at least two vertices and both $G$ and $H$ contain at least one loop, then $G\times H$ is not a median graph.
 
 \end{prop}
 
  \begin{proof}
  Let $u,v$ be adjacent vertices in $G$ and $x,y$ be adjacent vertices in $H$, and let both $u$ and $x$ have loops. Notice that $G\times H$ contains the triangle $(u,x),(v,x),(u,y)$. Since median graphs are bipartite by Lemma \ref{bipartite}, this means that $G \times H$ is not median.
    
   
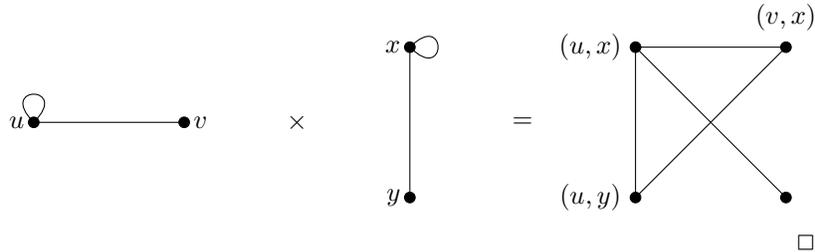
\begin{tikzpicture}
 
 \filldraw [black]
(0,1) circle (2pt)
(2,1) circle (2pt);

\draw (0,1) .. controls (.5,1.5) and(-.5,1.5) .. (0,1);
\draw (0,1) -- (2,1);
\coordinate[label=left:$u$] (u) at (0,1);
\coordinate[label=right:$v$] (v) at (2,1);
 
\node at (3.5,1)
{$\times$};
 
 \filldraw [black]
(5,0) circle (2pt)
(5,2) circle (2pt);

\draw (5,2) .. controls (5.5,2.5) and(5.5,1.5) .. (5,2);
\draw (5,0) -- (5,2);

\coordinate[label=left:$y$] (y) at (5,0);
\coordinate[label=left:$x$] (x) at (5,2);

\node at (6.5,1)
{$=$};

\filldraw [black]
(8,0) circle (2pt)
(8,2) circle (2pt)
(10,0) circle (2pt)
(10,2) circle (2pt);

\draw (8,2) --(10,2);
\draw (8,2) --(10,0);
\draw (10,2) --(8,0);
\draw (8,0) -- (8,2);

\node at (7.4,2) {$(u,x)$};
\node at (10,2.4) {$(v,x)$};
\node at (7.4,0) {$(u,y)$};

\end{tikzpicture}

\qed
\end{proof}
  
  \vspace{.2 in}
  
   Next, we define a few graphs.
 
\begin{itemize}
 \item Let $H_1$ be the graph on the vertex set $x,y,z$ with edges \linebreak $\{(x,y), (y,y), (y,z)\}$.
 \item Let $H_2$ be the graph on the vertex set $x,y$ with edges \linebreak $\{(x,x), (x,y), (y,y)\}$.
 \item Let $H_3$ be the graph on the vertex set $x,y,z$ with edges \linebreak
 $\{(x,x), (x,y), (y,z), (z,z)\}$.
 \item For any integer $k\geq 4$ let $H'_k$ be the graph on the vertex set $x_1, \dots, x_k$ \\with edges $\{(x_1,x_1), (x_1,x_2), (x_2,x_3), \dots, (x_{k-1},x_k), (x_k,x_k)\}$.
 \item For any integer $k\geq 3$ let $D_k$ be the graph on the vertex set $x_1, \dots, x_k$ \\with edges $\{(x_1,x_1), (x_1,x_2), (x_2,x_3), \dots, (x_{k-1},x_k), (x_k,x_1)\}$
 \item For any integer $k\geq 2$ let $E_k$ be the graph on the vertex set $x_1, \dots, x_k$ \\with edges $\{(x_1,x_1), (x_1,x_2), (x_2,x_3), \dots, (x_{k-1},x_k)\}$
 \end{itemize}

 \begin{tikzpicture}
\filldraw [black] 
(0,0) circle (2pt)
(0,1) circle (2pt)
(0,2) circle (2pt);

\draw (0,1) .. controls (.5,1.5) and (.5,.5) .. (0,1);
\draw (0,0) -- (0,2);

\filldraw [black]
(2,0) circle (2pt)
(2,2) circle (2pt);

\draw (2,0) .. controls (2.5,.5) and (2.5,-.5) .. (2,0);
\draw (2,2) .. controls (2.5,2.5) and (2.5,1.5) .. (2,2);
\draw (2,0) -- (2,2);

\filldraw [black]
(4,0) circle (2pt)
(4,1) circle (2pt)
(4,2) circle (2pt);

\draw (4,0) .. controls (4.5,.5) and (4.5,-.5) .. (4,0);
\draw (4,2) .. controls (4.5,2.5) and (4.5,1.5) .. (4,2);
\draw (4,0) -- (4,1) -- (4,2);

\filldraw [black]
(6,0) circle (2pt)
(6,1.5) circle (2pt)
(6,2) circle (2pt)
(6,1.125) circle (.5pt)
(6,1) circle (.5pt)
(6,.85) circle (.5pt);

\draw (6,2) .. controls (6.5,2.5) and (6.5,1.5) .. (6,2);
\draw (6,0) .. controls (6.5,.5) and (6.5,-.5) .. (6,0);
\draw (6,2) -- (6,1.25);
\draw (6,0) -- (6,.5);

\coordinate[label=left:$x_k$] (k) at (6,0);
\coordinate[label=left:$x_1$] (1) at (6,2);

\filldraw [black]
(8,0) circle (2pt)
(8,1.5) circle (2pt)
(8,2) circle (2pt)
(8,1.125) circle (.5pt)
(8,1) circle (.5pt)
(8,.85) circle (.5pt);

\draw (8,2) .. controls (8.5,2.5) and (8.5,1.5) .. (8,2);
\draw (8,0) .. controls (7.5,1) .. (8,2);
\draw (8,2) -- (8,1.25);
\draw (8,0) -- (8,.5);

\filldraw [black]
(10,0) circle (2pt)
(10,1.5) circle (2pt)
(10,2) circle (2pt)
(10,1.125) circle (.5pt)
(10,1) circle (.5pt)
(10,.85) circle (.5pt);

\draw (10,2) .. controls (10.5,2.5) and (10.5,1.5) .. (10,2);
\draw (10,2) -- (10,1.25);
\draw (10,0) -- (10,.5);

\node at (0,-1)
{$H_1$};
\node at (2,-1)
{$H_2$};
\node at (4,-1)
{$H_3$};
\node at (6,-1)
{$H'_k$};
\node at (8,-1)
{$D_k$};
\node at (10,-1)
{$E_k$};

\end{tikzpicture}

  Finally, we state some elementary but essential facts.

 \begin{prop} \label{odd cycle}
 If $G$ and $H$ are connected graphs on at least two vertices, G contains an odd cycle and $H$ contains at least one loop, then $G \times H$ is not a median graph.
 
 \end{prop}
 
 \begin{proof}
 Let $v_1,\dots v_k$ be an odd cycle in $G$ and let $x$ be a vertex with a loop in $H$. Observe that $(v_1,x),(v_2,x), \dots, (v_k,x)$ is an odd cycle in $G \times H$. Since median graphs are bipartite by Lemma \ref{bipartite}, this means that $G \times H$ is not median.
  \qed
 
 \end{proof}
 
 \begin{prop}\label{cycleloop}
 If $G$ is a connected graph containing $P_3$ as a subgraph and $H$ is a connected graph containing $D_k$ as an induced subgraph for some $k>2$, then $G\times H$ is not a median graph.
 
 \end{prop}
 
 \begin{proof}
Label the vertices of the path $P_3 \subseteq G$ as $v_1,v_2,v_3$, where $v_i \sim v_{i+1}$ for $i=1,2$. Label the vertices of $D_k\subseteq H$ as $x_1,x_2,\dots x_k$, where $x_1 \sim x_1$, $x_1 \sim x_k$, and $x_i \sim x_{i+1}$ for $1\leq i \leq k-1$. Notice that by Proposition \ref{loops in G and H}, if $v_2$ has a loop in $G$, then we are done, so assume that $d_{G\times H}((v_2,x_1),(v_2,x_2))$$=d_{G\times H}((v_2,x_1),(v_2,x_k))$$=d_{G\times H}((v_2,x_2),(v_2,x_k))$\linebreak $=2$. Therefore, $(v_1,x_1)$ and $(v_3,x_1)$ are median vertices for the triple, $(v_2,x_1),(v_2,x_2),(v_2,x_k)$, so that $G\times H$ is not a median graph.

 
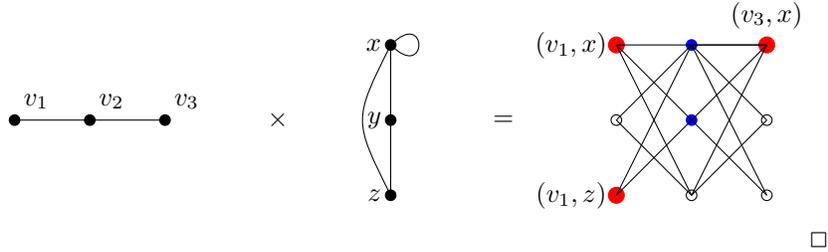
\begin{tikzpicture}
 \filldraw [black]
(0,1) circle (2pt)
(1,1) circle (2pt)
(2,1) circle (2pt);

\draw (0,1) -- (1,1);
\draw (1,1) -- (2,1);
\coordinate[label=right:$v_1$] (1) at (0,1.25);
\coordinate[label=right:$v_2$] (2) at (1,1.25);
\coordinate[label=right:$v_3$] (3) at (2,1.25);
 
 \node at (3.5,1)
{$\times$};
 
 \filldraw [black]
(5,0) circle (2pt)
(5,1) circle (2pt)
(5,2) circle (2pt);

\draw (5,2) .. controls (5.5,2.5) and (5.5,1.5) .. (5,2);
\draw (5,0) .. controls (4.5,1) .. (5,2);
\draw (5,2) -- (5,1);
\draw (5,0) -- (5,1);
\coordinate[label=left:$z$] (3) at (5,0);
\coordinate[label=left:$y$] (2) at (5,1);
\coordinate[label=left:$x$] (1) at (5,2);
 
\node at (6.5,1)
{$=$};

\draw [black]
(8,1) circle (2pt)
(9,0) circle (2pt)
(10,0) circle (2pt)
(10,1) circle (2pt);

\filldraw [blue]
(9,1) circle (2pt)
(9,2) circle (2pt);
\filldraw [red]
(10,2) circle (3pt)
(8,2) circle (3pt)
(8,0) circle (3pt);

\draw (8,2) --(10,2);
\draw (8,2) --(9,1);
\draw (8,2) -- (9,0);
\draw (9,2) --(8,1);
\draw (9,2) -- (10,1);
\draw (9,2) --(8,0);
\draw (9,2) -- (10,0);
\draw (9,2) --(10,2);
\draw (10,2) -- (9,1);
\draw (10,2) --(9,0);
\draw (8,1) -- (9,0);
\draw (9,1) -- (8,0);
\draw (9,1) --(10,0);
\draw (10,1) -- (9,0);

\node at (10,2.4) {$(v_3,x)$};
\node at (7.4,2) {$(v_1,x)$};
\node at (7.4,0) {$(v_1,z)$};

\end{tikzpicture}
 
 \qed
 \end{proof}

\section{Two Loops}
 In this section we prove the following:
 
 \begin{theo}\label{two loops}
 If $G$ is a connected graph with at least three vertices and $H$ is a connected graph with at least two vertices and contains more than one loop, then the direct product $G\times H$ is not a median graph.
 
 \end{theo}
 
 The following observations exclude the cases when two loops lie close together.

%The next result was proven by Munarini in \cite{Munarini}.

\begin{prop} \label{H2}
If $G$ is a connected graph containing $P_3$ as a subgraph and $H$ is a connected graph containing $H_2$, then $G\times H$ is not a median graph.

\end{prop}

\begin{proof}
Label the vertices of the path $P_3 \subseteq G$ as $v_1,v_2,v_3$, where $v_i \sim v_{i+1}$ for $i=1,2$. Label the vertices of $H_2 \subseteq H$ as $x,y$ where $x \sim y$, $x \sim x$, and $y \sim y$. By Proposition \ref{K23sub}, it is enough to notice that $G \times H$ contains a $K_{2,3}$ subgraph. %By Proposition \ref{odd cycle}, if $v_1 \sim v_3$ we would be done, so we assume to the contrary. Similarly, we can assume that $v_3 \not \sim v_3$ by Proposition \ref{loops in G and H}. By these observations, we conclude that $P_3$ is convex in $G$, $H_2$ is convex in $H$, and $d_{G\times H}((v_1,x),(v_3,x))=d_{G\times H}((v_1,x),(v_3,y))=d_{G\times H}((v_3,x),(v_3,y))=2$. Notice that $(v_2,x)$ and $(v_2,y)$ are median vertices in $P_3 \times H_2$. By convexity of $P_3$ and $H_2$, we conclude that there is no unique median in $G\times H$.

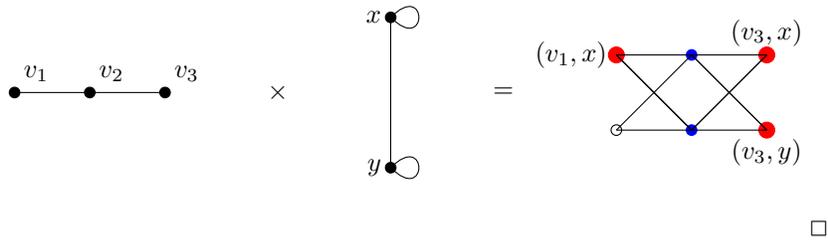
\begin{tikzpicture}
 \filldraw [black]
(0,1) circle (2pt)
(1,1) circle (2pt)
(2,1) circle (2pt);

\draw (0,1) -- (1,1);
\draw (1,1) -- (2,1);
\coordinate[label=right:$v_1$] (1) at (0,1.25);
\coordinate[label=right:$v_2$] (2) at (1,1.25);
\coordinate[label=right:$v_3$] (3) at (2,1.25);
 
 \node at (3.5,1){$\times$};
 
 \filldraw [black]
(5,0) circle (2pt)
(5,2) circle (2pt);

\draw (5,2) .. controls (5.5,2.5) and (5.5,1.5) .. (5,2);
\draw (5,0) .. controls (5.5,.5) and (5.5,-.5) .. (5,0);
\draw (5,0) -- (5,2);
\coordinate[label=left:$y$] (3) at (5,0);
\coordinate[label=left:$x$] (1) at (5,2);
 
\node at (6.5,1) {$=$};

\draw [black]
(8,.5) circle (2pt);

\filldraw [blue]
(9,1.5) circle (2pt)
(9,.5) circle (2pt);

\filldraw [red]
(8,1.5) circle (3pt)
(10,1.5) circle (3pt)
(10,.5) circle (3pt);

\draw (8,1.5) --(10,1.5);
\draw (8,.5) -- (10,.5);
\draw (8,1.5) --(9,.5);
\draw (8,.5) -- (9,1.5);
\draw (9,.5) --(8,1.5);
\draw (9,1.5) -- (10,.5);
\draw (9,.5) --(8,1.5);
\draw (9,1.5) -- (10,.5);
\draw (9,.5) --(10,1.5);
\draw (10,1.5) -- (9,.5);

\node at (10,1.8) {$(v_3,x)$};
\node at (7.4,1.5) {$(v_1,x)$};
\node at (10,.2) {$(v_3,y)$};

\end{tikzpicture}

\qed

\end{proof}

\begin{prop} \label{H3}
If $G$ is a connected graph containing $P_3$ as a subgraph and $H$ is a connected graph containing $H_3$, then $G\times H$ is not a median graph.

\end{prop}

\begin{proof}
Label the vertices of the path $P_3 \subseteq G$ as $v_1,v_2,v_3$, where $v_i \sim v_{i+1}$ for $i=1,2$. Label the vertices of $H_3 \subseteq H$ as $x,y,z$ where $x \sim y \sim z$, $x \sim x$, and $z \sim z$.
By Proposition \ref{odd cycle} $v_1 \not \sim v_3$ and hence $d((v_1,x),(v_3,y))\geq 2$. For the same reason, $d((v_1,z),(v_3,y))\geq 2$.

By Proposition \ref{cycleloop} we may assume $x \not \sim z$ in $H$. Suppose there exist vertices $g \in G$ and $h \in H$ so that $(g,h)\sim (v_1,x)$, $(g,h)\sim (v_3,y)$, $(g,h)\sim (v_1,z)$.
%Again by Proposition \ref{odd cycle} $N(v_1)\cap N(v_2) = \emptyset$, which means that there is no vertex $g\in G$ so that $v_1 \sim g\sim v_2$.
If $h=x$, then $(g,h)\not \sim (v_1,z)$. If $h=y$, then $(g,h)\not \sim (v_3,y)$. If $h=z$, then $(g,h)\not \sim (v_1,x)$. We are left with the case when $h$ is adjacent to $x,y,$ and $z$, but this is impossible by Proposition \ref{cycleloop}. Hence, $P_3$ is convex in $G$ and $H_3$ is convex in $H$.

By inspection, we see that there is no median vertex in $P_3 \times H_3$ for \linebreak 
$(v_1,x),(v_3,y), (v_1,z)$, and by the previous observation, there can be no median in $G\times H$.

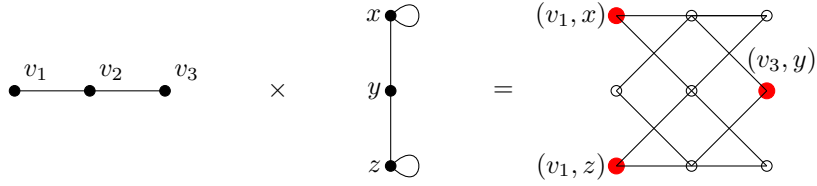
\begin{tikzpicture}
 \filldraw [black]
(0,1) circle (2pt)
(1,1) circle (2pt)
(2,1) circle (2pt);

\draw (0,1) -- (1,1);
\draw (1,1) -- (2,1);
\coordinate[label=right:$v_1$] (1) at (0,1.25);
\coordinate[label=right:$v_2$] (2) at (1,1.25);
\coordinate[label=right:$v_3$] (3) at (2,1.25);
 
 \node at (3.5,1)
{$\times$};
 
 \filldraw [black]
(5,0) circle (2pt)
(5,1) circle (2pt)
(5,2) circle (2pt);

\draw (5,2) .. controls (5.5,2.5) and (5.5,1.5) .. (5,2);
\draw (5,0) .. controls (5.5,.5) and (5.5,-.5) .. (5,0);
\draw (5,2) -- (5,1);
\draw (5,0) -- (5,1);
\coordinate[label=left:$z$] (3) at (5,0);
\coordinate[label=left:$y$] (2) at (5,1);
\coordinate[label=left:$x$] (1) at (5,2);
 
\node at (6.5,1)
{$=$};

\draw [black]
(8,1) circle (2pt)
(9,0) circle (2pt)
(10,0) circle (2pt)
(10,2) circle (2pt)
(9,1) circle (2pt)
(9,2) circle (2pt);

\filldraw [red]
(10,1) circle (3pt)
(8,2) circle (3pt)
(8,0) circle (3pt);

\draw (8,2) --(10,2);
\draw (8,0) -- (10,0);
\draw (8,2) --(9,1);
\draw (9,2) --(8,1);
\draw (9,2) -- (10,1);
\draw (9,2) --(10,2);
\draw (10,2) -- (9,1);
\draw (8,1) -- (9,0);
\draw (9,1) -- (8,0);
\draw (9,1) --(10,0);
\draw (10,1) -- (9,0);

\node at (10.2,1.4) {$(v_3,y)$};
\node at (7.4,2) {$(v_1,x)$};
\node at (7.4,0) {$(v_1,z)$};

\end{tikzpicture}

\qed

\end{proof}

\begin{lemma} \label{H'k}
If $G$ is a connected graph containing $P_3$ as a subgraph and $H$ is a connected graph with isometric subgraph $H'_k$ for some $k>1$, then $G\times H$ is not a median graph.

\end{lemma}

\begin{proof}
Label the vertices of the path $P_3 \subseteq G$ as $v_1,v_2,v_3$, where $v_i \sim v_{i+1}$ for $i=1,2$. Label the vertices of $H'_k \subseteq H$ as $x_1,\dots,x_k$ where $x_1\sim x_1, x_k \sim x_k,$ and $x_i\sim x_{i+1}$ for $1\leq i \leq k-1$. We apply Proposition \ref{odd cycle}, and assume $v_1 \not \sim v_3$. Since $H'_k$ is isometric in $H$, $d_H(x_1,x_k)=k-1$, which means $d_{G\times H}((v_1,x_1),(v_3,x_k))\geq k-1$.\\

\emph{Case 1}: Suppose $k$ is even.\\

Consider the paths 
\[\cP_1=(v_1,x_1)\sim (v_2,x_2)\sim (v_1,x_3) \sim \dots \sim (v_2,x_k) \sim (v_1,x_k)\]
 and 
\[\cP_2=(v_1,x_1)\sim (v_2,x_1) \sim (v_1,x_2) \sim (v_2,x_3) \sim \dots \sim (v_2,x_{k-1}) \sim (v_1,x_k)\]

Observe that $\cP_1$ contains $(v_2,x_{k-1})$ and $\cP_2$ contains $(v_2, x_k)$, and both paths are isometric as they are of length $k-1$.\

Likewise, consider the paths
\[\cP_3=(v_1,x_1)\sim (v_2,x_2)\sim (v_3,x_3) \sim (v_2,x_4) \sim (v_3,x_5) \sim \dots \]\[\sim (v_2,x_k)\sim (v_3,x_k)\]
 and 
\[\cP_4=(v_1,x_1)\sim (v_2,x_1) \sim (v_3,x_2) \sim (v_2,x_3) \sim \dots \sim (v_2,x_{k-1}) \sim (v_3,x_k)\]

Observe that $\cP_3$ contains $(v_2,x_k)$ and $\cP_4$ contains $(v_2, x_{k-1})$, and both paths are isometric as they are of length $k-1$.\

Notice that $d((v_1,x_k),(v_3,x_k))=2$ and that both $(v_2,x_k)$ and $(v_2, x_{k-1})$ lie on geodesics from $(v_1,x_k)$ to $(v_3,x_k)$.
Therefore there is no unique median vertex for $(v_1,x_1),(v_1,x_k),(v_3,x_k)$ and $G \times H$ is not a median graph.\\

\emph{Case 2}: Suppose $k$ is odd.\\

Consider the paths 
\[\cQ_1=(v_2,x_1)\sim (v_1,x_2)\sim (v_2,x_3) \sim \dots \sim (v_2,x_k) \sim (v_1,x_k)\]
 and 
\[\cQ_2=(v_2,x_1)\sim (v_1,x_1) \sim (v_2,x_2) \sim (v_1,x_3) \sim \dots \sim (v_2,x_{k-1}) \sim (v_1,x_k)\]

Observe that $\cQ_1$ contains $(v_2,x_k)$ and $\cQ_2$ contains $(v_2, x_{k-1})$. Both paths are of length $k$. Since $d(x_1,x_k)=k-1$, $d((v_2,x_1),(v_1,x_k))\geq k-1$ and $d((v_2,x_1),(v_3,x_k))\geq k-1$. Suppose there exists a path $\cQ$ from $(v_2,x_1)$ to $(v_1,x_k)$ of length $k-1$. By Proposition \ref{findoddcycle}, $\cQ_1 \cup \cQ$ contains an odd cycle. However, in this case $G \times H$ is not a median graph by Lemma \ref{bipartite}.\

Likewise, consider the paths
\[\cQ_3=(v_2,x_1)\sim (v_3,x_2)\sim (v_2,x_3) \sim \dots \sim (v_2,x_k) \sim (v_3,x_k)\]
 and 
\[\cQ_4=(v_2,x_1)\sim (v_3,x_1) \sim (v_2,x_2) \sim \dots \sim (v_2,x_{k-1}) \sim (v_3,x_k)\]

Observe that $\cQ_3$ contains $(v_2,x_k)$ and $\cQ_4$ contains $(v_2, x_{k-1})$. Both paths are of length $k$ and isometric by the same argument used above for $\cQ_1$ and $\cQ_2$.\

Notice that $d((v_1,x_k),(v_3,x_k))=2$ and that both $(v_2,x_k)$ and $(v_2, x_{k-1})$ lie on geodesics from $(v_1,x_k)$ to $(v_3,x_k)$.
Therefore there is no unique median vertex for $(v_1,x_1),(v_1,x_k),(v_3,x_k)$ and $G \times H$ is not a median graph.\\

\qed

\end{proof}

We are now ready to prove our Theorem.

\vspace{.2 in}

\begin{proof} (Theorem \ref{two loops})
Choose two vertices $x$ and $y$ of $H$ with loops. Label the vertices of the path $P_3 \subseteq G$ as $v_1,v_2,v_3$, where $v_i \sim v_{i+1}$ for $i=1,2$. Let $\cP$ be a shortest path from $x$ to $y$ in $H$. Notice that $\cP$ is isometric in $H$. If $|\cP|$ is odd, then we are done by Proposition \ref{H2} and Lemma \ref{H'k}. If $|\cP|$ is even, then we are done by Proposition \ref{H3} and Lemma \ref{H'k}.

\qed
\end{proof}

\section{One Loop}

The following result limits the placement of loops in $G$.

\begin{prop} \label{H1}
If $G$ is a connected graph containing $P_3$ as a subgraph and $H$ is a connected graph containing $H_1$ as a subgraph, then $G\times H$ is not a median graph.

\end{prop}

\begin{proof}
Label the vertices of the path $P_3 \subseteq G$ as $v_1,v_2,v_3$, where $v_i \sim v_{i+1}$ for $i=1,2$. Label the vertices of $H_1 \subseteq H$ as $x,y,z$ where $x \sim y \sim z$, $y \sim y$. By Proposition \ref{K23sub}, it is enough to notice that the subgraph of $G \times H$ on the vertices $(v_2,x),(v_2,y),(v_2,z),(v_1,y),(v_3,y)$ is $K_{2,3}$. %By Proposition \ref{odd cycle} we can assume $v_1 \not \sim v_3$ and $N(v_1,v_2) = \emptyset$. By Proposition \ref{cycleloop} we can assume that the only path from $x$ to $y$ is $x\sim y$, from $y$ to $z$ is $y\sim z$, and from $x$ to $z$ is $x\sim y \sim z$. By Proposition \ref{H2}, we can assume $x\not \sim x$. Thus, $d((v_1,x),(v_2,x))= d((v_1,x),(v_2,z))=3$ and $d((v_2,x),(v_2,z))=2$. Hence, $\{(v_1,y),(v_2,y),(v_3,y)\}$ lie on geodesics from $(v_1,x)$ to $(v_2,x)$, $\{(v_1,y),(v_2,y),(v_3,y)\}$ lie on geodesics from $(v_1,x)$ to $(v_2,z)$, and $\{(v_1,y),(v_3,y)\}$ lie on geodesics from $(v_2,x)$ to $(v_2,z)$. Therefore $(v_1,y)$ and $(v_3,y)$ are median vertices for $(v_1,x),(v_2,x),(v_2,z)$, and $G\times H$ is not a median graph.

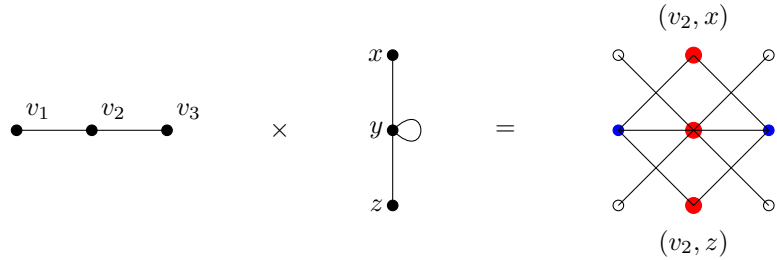
\begin{tikzpicture}
 \filldraw [black]
(0,1) circle (2pt)
(1,1) circle (2pt)
(2,1) circle (2pt);

\draw (0,1) -- (1,1);
\draw (1,1) -- (2,1);
\coordinate[label=right:$v_1$] (1) at (0,1.25);
\coordinate[label=right:$v_2$] (2) at (1,1.25);
\coordinate[label=right:$v_3$] (3) at (2,1.25);
 
 \node at (3.5,1)
{$\times$};
 
 \filldraw [black]
(5,0) circle (2pt)
(5,1) circle (2pt)
(5,2) circle (2pt);

\draw (5,1) .. controls (5.5,1.5) and (5.5,.5) .. (5,1);
\draw (5,2) -- (5,0);
\coordinate[label=left:$z$] (3) at (5,0);
\coordinate[label=left:$y$] (2) at (5,1);
\coordinate[label=left:$x$] (1) at (5,2);
 
\node at (6.5,1)
{$=$};

\draw [black]
(9,0) circle (2pt)
(10,0) circle (2pt)
(10,2) circle (2pt)

(9,2) circle (2pt)
(8,2) circle (2pt)
(8,0) circle (2pt);

\filldraw [blue]
(8,1) circle (2pt)
(10,1) circle (2pt);

\filldraw [red]
(9,2) circle (3pt)
(9,1) circle (3pt)
(9,0) circle (3pt);

\draw (8,1) --(10,1);
\draw (8,2) --(9,1);
\draw (9,2) --(8,1);
\draw (9,2) -- (10,1);
\draw (10,2) -- (9,1);
\draw (8,1) -- (9,0);
\draw (9,1) -- (8,0);
\draw (9,1) --(10,0);
\draw (10,1) -- (9,0);

\node at (9,2.5) {$(v_2,x)$};
%\node at (7.4,2) {$(v_1,x)$};
\node at (9,-.5) {$(v_2,z)$};

\end{tikzpicture}

\qed

\end{proof}

The next proposition further limits the structure of median product graphs.

\begin{prop}\label{star}
Let $G$ be a connected graph on at least three vertices and $H$ a connected graph containing $E_3$ as a subgraph. If either $G$ or $H$ contain the star $S_3$ as subgraph, then $G\times H$ is not a median graph.

\end{prop}

\begin{proof}
Label the vertices of $P_3$ in $G$ by $v_1,v_2,v_3$. Suppose $H$ contains $S_3$ as a subgraph and label the vertices of $S_3$ by $x_1,x_2,x_3,x_4$ with $x_1$ as central vertex. By Proposition \ref{K23sub}, it is enough to notice that the subgraph of $G \times H$ on the vertices $(v_2,x_3), (v_2,x_2), (v_2,x_4), (v_1,x_1), (v_3,x_1)$ is $K_{2,3}$.
 %Notice that by Proposition \ref{cycleloop}, $x_2, x_3, x_4$ form an independent set of vertices.
%By Proposition \ref{loops in G and H}, we can assume there is no loop at $v_2$. However, this means that $d((v_2,x_2),(v_2,x_3))=d((v_2,x_3),(v_2,x_4))=d((v_2,x_3),(v_2,x_4))=2$. Notice $(v_1,x_1)$ and $(v_3,x_1)$ are both on these length two paths and serve as median vertices. Therefore, $G \times H$ is not median.

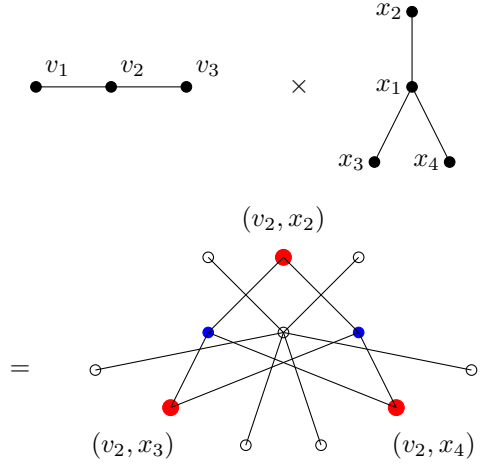
\begin{center}

\begin{tikzpicture}
 \filldraw [black]
(0,1) circle (2pt)
(1,1) circle (2pt)
(2,1) circle (2pt);

\draw (0,1) -- (2,1);
\coordinate[label=right:$v_1$] (1) at (0,1.25);
\coordinate[label=right:$v_2$] (2) at (1,1.25);
\coordinate[label=right:$v_3$] (3) at (2,1.25);
 
 \node at (3.5,1)
{$\times$};
 
 \filldraw [black]
(5,2) circle (2pt)
(5,1) circle (2pt)
(5.5,0) circle (2pt)
(4.5,0) circle (2pt);

\draw (5,2) -- (5,1);
\draw (5,1) -- (5.5,0);
\draw (5,1) -- (4.5,0);
\coordinate[label=left:$x_3$] (3) at (4.5,0);
\coordinate[label=left:$x_4$] (4) at (5.5,0);
\coordinate[label=left:$x_2$] (2) at (5,2);
\coordinate[label=left:$x_1$] (1) at (5,1);

\end{tikzpicture}
\end{center}

\begin{center}
\begin{tikzpicture}
 
\node at (6.5,1)
{$=$};

\draw [black]
(9,2.5) circle (2pt)
(11,2.5) circle (2pt)
(10,1.5) circle (2pt)
(7.5,1) circle (2pt)
(9.5,0) circle (2pt)
(10.5,0) circle (2pt)
(12.5,1) circle (2pt);

\filldraw [blue]
(9,1.5) circle (2pt)
(11,1.5) circle (2pt);

\filldraw [red]
(10,2.5) circle (3pt)
(8.5,.5) circle (3pt)
(11.5,.5) circle (3pt);

\draw (9,2.5) -- (10,1.5);
\draw (10,2.5) --(9,1.5);
\draw (11,2.5) --(10,1.5);
\draw (10,2.5) -- (11,1.5);

\draw (10,1.5) -- (7.5,1);
\draw (10,1.5) -- (9.5,0);
\draw (9,1.5) -- (8.5,.5);
\draw (11,1.5) -- (8.5,.5);

\draw (10,1.5) --(10.5,0);
\draw (10,1.5) -- (12.5,1);
\draw (9,1.5) -- (11.5,.5);
\draw (11,1.5) -- (11.5,.5);

\node at (10,3) {$(v_2,x_2)$};
\node at (8,0) {$(v_2,x_3)$};
\node at (12,0) {$(v_2,x_4)$};

\end{tikzpicture}
\end{center}

\vspace{.2 in}

Label the vertices of $E_3$ in $H$ by $x_1,x_2,x_3$. Suppose $G$ contains $S_3$ as a subgraph and label the vertices of $S_3$ by $v_1,v_2,v_3,v_4$ with $v_1$ as central vertex. Again, by Proposition \ref{K23sub}, we see that the subgraph on the vertices $(v_2,x_2),(v_3,x_2),(v_4,x_2),(v_1,x_3),(v_1,x_1)$ is $K_{2,3}$.
%By Proposition \ref{H2}, we can assume that there is no loop at $x_2$. As before, $d((v_2,x_2),(v_3,x_2))=d((v_2,x_2),(v_4,x_2))=d((v_3,x_2),(v_4,x_2))=2$. Notice $(v_1,x_1)$ and $(v_1,x_3)$ are both on these length two paths and serve as median vertices. Therefore, $G \times H$ is not median.

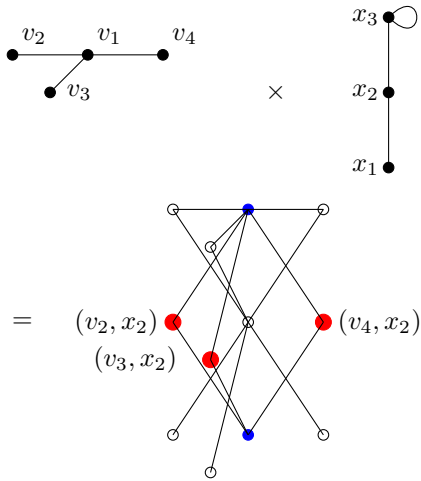
\begin{center}

\begin{tikzpicture}
 \filldraw [black]
(0,1.5) circle (2pt)
(1,1.5) circle (2pt)
(2,1.5) circle (2pt)
(.5,1) circle (2pt);

\draw (0,1.5) -- (2,1.5);
\draw (.5,1) -- (1,1.5);
\coordinate[label=right:$v_2$] (2) at (0,1.75);
\coordinate[label=right:$v_1$] (1) at (1,1.75);
\coordinate[label=right:$v_4$] (3) at (2,1.75);
\coordinate[label=right:$v_3$] (4) at (.6,1);
 
 \node at (3.5,1)
{$\times$};
 
 \filldraw [black]
(5,0) circle (2pt)
(5,1) circle (2pt)
(5,2) circle (2pt);

\draw (5,2) .. controls (5.5,2.5) and (5.5,1.5) .. (5,2);
\draw (5,2) -- (5,0);
\coordinate[label=left:$x_1$] (3) at (5,0);
\coordinate[label=left:$x_2$] (2) at (5,1);
\coordinate[label=left:$x_3$] (1) at (5,2);

\end{tikzpicture}
\end{center}
 
\begin{center} 
\begin{tikzpicture}
\node at (6.5,1)
{$=$};

\draw [black]
(8.5,2.5) circle (2pt)
(10.5,2.5) circle (2pt)
(9,2) circle (2pt)
(9.5,1) circle (2pt)
(8.5,-.5) circle (2pt)
(10.5,-.5) circle (2pt)
(9,-1) circle (2pt);

\filldraw [red]
(8.5,1) circle (3pt)
(9,.5) circle (3pt)
(10.5,1) circle (3pt);

\filldraw [blue]
(9.5,2.5) circle (2pt)
(9.5,-.5) circle (2pt);

\draw (8.5,2.5) -- (10.5,2.5);
\draw (9,2) --(9.5,2.5);
\draw (8.5,1) --(9.5,2.5);
\draw (9.5,1) -- (9,2);
\draw (10.5,1) --(9.5,2.5);
\draw (10.5,2.5) --(9.5,1);
\draw (8.5,1) --(9.5,-.5);
\draw (9.5,1) -- (8.5,-.5);
\draw (9.5,1) --(10.5,-.5);
\draw (9.5,1) -- (9,-1);
\draw (9.5,1) -- (8.5,2.5);
\draw (10.5,1) -- (9.5,-.5);
\draw (9.5,2.5) -- (9,.5);
\draw (9,.5) -- (9.5,-.5);

\node at (7.75,1) {$(v_2,x_2)$};
\node at (8,.5) {$(v_3,x_2)$};
\node at (11.25,1) {$(v_4,x_2)$};

\end{tikzpicture}
\end{center}

\qed
\end{proof}

The next result can be immediately deduced from Lemma 2.2 in \cite{BJKZ}.

\begin{prop}\label{cycle}
Let $G$ be a connected graph on at least three vertices and $H$ a connected graph containing $E_3$ as a subgraph. For any integer $k\geq 3$, if either $G$ or $H$ contain the cycle $C_k$ as subgraph, then $G\times H$ is not a median graph.

\end{prop}

%\begin{proof}

%\qed
%\end{proof}

The following result can be found in \cite{KlavzarSkrekovki}.

\begin{theo}\label{KS}

A connected grid graph with $n$ vertices and $m$ edges is a median graph if and only if it contains $m-n+1$ squares.
\end{theo}

\begin{corol} \label{one loop}
For any integers $k\geq 3$ and $l\geq 2$, $P_k\times E_l$ is a median graph.
\end{corol}

\begin{proof}
Label the vertices of $P_k$ by $v_1, \dots, v_k$ and those of $E_l$ by $x_1,\dots, x_l$. There are $2k-2$ edges incident to vertices with an $x_i$ entry in the second coordinate and those with $x_{i+1}$ in the second coordinate, for $1\leq i \leq l-1$. Summing over all $i$, we get $(l-1)(2k-2)$. We add this sum to the $k-1$ edges between vertices with an $x_1$ in the second entry in the second coordinate, and find the total number of edges is $m=(2l-1)(k-1)$. To find the number of squares we ``unfold" the graph by drawing the vertices as follows:

 Call the vertices with index $1$ in the second coordinate and odd sum of indices of entries in both coordinates, the \emph{middle row}. Place the vertices of the middle row from left to right, increasing by the index of entries in the first coordinate. If the sum of indices of the entries of both coordinates is odd, then we place them above the middle row, increasing from the middle row up with the value of the index in the second coordinate and increasing from left to right with the value of the index in the first coordinate. If the sum of indices of the entries of both coordinates is even, then we place them below the middle row, increasing from the middle row down with the value of the index in the second coordinate and increasing from left to right with the value of the index in the first coordinate.

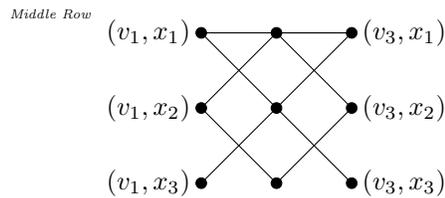
\vspace{.2 in}

\begin{center}
\begin{tikzpicture}

\filldraw [black]
(1,1) circle (2pt)
(2,0) circle (2pt)
(3,0) circle (2pt)
(3,1) circle (2pt)
(2,1) circle (2pt)
(2,2) circle (2pt)
(3,2) circle (2pt)
(1,2) circle (2pt)
(1,0) circle (2pt);

\draw (1,2) --(3,2);
\draw (1,2) --(2,1);
\draw (2,2) --(1,1);
\draw (2,2) -- (3,1);
\draw (2,2) --(3,2);
\draw (3,2) -- (2,1);
\draw (1,1) -- (2,0);
\draw (2,1) -- (1,0);
\draw (2,1) --(3,0);
\draw (3,1) -- (2,0);

\node at (3.7,2) {$(v_3,x_1)$};
\node at (0.3,2) {$(v_1,x_1)$};
\node at (-1,2.25) {\tiny\emph{Middle Row}};
\node at (3.7,1) {$(v_3,x_2)$};
\node at (.3,1) {$(v_1,x_2)$};
\node at (0.3,0) {$(v_1,x_3)$};
\node at (3.7,0) {$(v_3,x_3)$};

\end{tikzpicture}
\end{center}

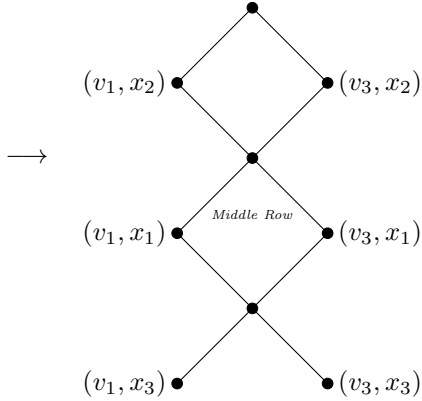
\begin{center}
\begin{tikzpicture}

\node at (5,1) {$\longrightarrow$};

\filldraw [black]
(8,3) circle (2pt)
(7,2) circle (2pt)
(9,2) circle (2pt)
(8,1) circle (2pt)
(7,0) circle (2pt)
(9,0) circle (2pt)
(8,-1) circle (2pt)
(7,-2) circle (2pt)
(9,-2) circle (2pt);

\draw (8,3) -- (7,2);
\draw (8,3) -- (9,2);
\draw (7,2) -- (8,1);
\draw (9,2) -- (8,1);
\draw (8,1) -- (7,0);
\draw (8,1) -- (9,0);
\draw (7,0) -- (8,-1);
\draw (9,0) -- (8,-1);
\draw (8,-1) -- (7,-2);
\draw (8,-1) -- (9,-2);

\node at (8,.25) {\tiny\emph{Middle Row}};
\node at (6.3,0) {$(v_1,x_1)$};
\node at (9.7,0) {$(v_3,x_1)$};
\node at (6.3,2) {$(v_1,x_2)$};
\node at (9.7,2) {$(v_3,x_2)$};
\node at (6.3,-2) {$(v_1,x_3)$};
\node at (9.7,-2) {$(v_3,x_3)$};

\end{tikzpicture}
\end{center}

\vspace{.2 in}

This representation easily allows us to count the number of squares as $(k-2) \times (l-1)$, and since the number of squares is equal to $m-n+1$, we are done by Theorem \ref{KS}.

\qed

\end{proof}

\begin{theo}
If $G$ is either a connected graph on at least three vertices or a connected graph with at least one loop, and $H$ is a connected graph on at least two vertices with at least one loop, then the only median graphs of the form $G\times H$ occurs when $G=P_k$ and $H = P_l$ with a loop at an end vertex, for any integers $k\geq 3, l\geq 2$.
\end{theo}

\begin{proof}
If $H$ has two loops, then it is not median by Theorem \ref{two loops}. If $H$ has one loop, say at vertex $v$, then by Proposition \ref{H1}, $deg(v)\leq 3$, and $v$ is an end vertex. If $v$ lies on a path to a vertex of degree at least three, then by Proposition \ref{star}, $G\times H$ is not median. If $G$ has a loop, then by Proposition \ref{loops in G and H}, $G\times H$ is not median. If $G$ has a vertex of degree at least three, then by Proposition \ref{star}, $G\times H$ is not median. If either $G$ or $H$ have cycles, then $G\times H$ is not median by Proposition \ref{cycle}. By Corollary \ref{one loop}, the theorem is proved.

\qed
\end{proof}

The remaining classifications of median product graphs, $G\times H$, consist of the case when $G=P_2$, which was partly done in \cite{BJKZ}. A possible future direction could be in attempting to solve the above problem with the added assumption that $H$ contains at least one loop.

\section{Acknowledgements}

The authors would like to thank the anonymous referee for the helpful comments.

\bibliographystyle{plain}
\bibliography{median}

 \end{document}